\documentclass[a4paper,12pt]{article}
\usepackage{amssymb,enumerate}
\usepackage{amsmath,amssymb}
\usepackage{amsfonts}
\usepackage{amsthm}
\newtheorem{thm}{Theorem}
\newtheorem{df}{Definition}

\newtheorem{cor}{Corollary}

\title{Fixed points and the stability of the linear functional equations in a single variable}

\author{Liviu C\u adariu and Laura Manolescu}
\date{}

\begin{document}

\maketitle

\begin{minipage}{120mm}
\small{\bf Abstract.} { We prove that an interesting result concerning generalized Hyers-Ulam-Rassias stability of a linear functional equation obtained in 2014 by S.M. Jung, D. Popa and M.T. Rassias in Journal of Global Optimization is a particular case of a fixed point theorem given by us in 2012. Moreover, we give a characterization of functions that can be approximated with a given error, by the solution of the previously mention linear equation. 
}\\



\end{minipage}
\section{Introduction}

In 1940, S.M. Ulam \cite{Ulam} raised the following question concerning the stability group homomorphisms while presenting a talk at the University of Wisconsin:
"When a solution of an equation differing slightly from a given one must be somehow
near to the solution of the given equation?” In a more precise formulation, the problem reads as follows:\\

Let $(G_1,\circ)$ be a group, $(G_2,\ast)$ be a metric group with the metric $d(\cdot,\cdot)$ and $\varepsilon>0.$ If there exists a $\delta>0$ such that $f:G_1\rightarrow G_2$ satisfies
 $$d(f(x\circ y),f(x)\ast f(y))\leq\delta,\quad\textrm{for all}~x,y\in G_1$$
 then there exists a homomorphism $h:G_1\rightarrow G_2$ with $$d(f(x),h(x))\leq\varepsilon,\quad\textrm{for all}~x\in G_1?$$

A first answer to Ulam's question, concerning the Cauchy equation, was given by D. H. Hyers \cite{Hyers1941} in 1941:

Let $E_1$ and $E_2$ be Banach spaces and let $f : E_1\rightarrow E_2$ be a mapping such that

$$\|f(x+y) - f(x) - f(y)\|\leq \delta$$

for all $ x, y\in E_1$ and $\delta > 0 $, that is, f is $\delta$-additive. Then the limit $$T(x)=\lim_{n\rightarrow\infty}\frac{f(2^nx)}{2^n}$$ exists for each $x\in E_1$ and $T:E_1\rightarrow E_2$ is the unique additive mapping such that $$\|f(x)-T(x)\|\leq\delta,\quad\textrm{for every}~x\in E_1.$$
Moreover, if $f(tx)$ is continuous in $t$ for each fixed $x\in E_1,$ then the function $T$ is linear. 

So, nowadays we speak about the Hyers-Ulam stability. Afterwards, different generalizations of that initial answer of
Hyers were obtained. Hyers’ theorem was generalized by Aoki \cite{Aoki1950} for additive mappings and 
Th.M. Rassias \cite{TRassias} generalized the theorem of Hyers for approximately linear mappings, by considering an unbounded Cauchy
difference. See also \cite{rassiasJohn} and \cite{TRassias1}.

 A further generalization was obtained by P. G\u{a}vru\c{t}a \cite{Gavruta1994} in 1994. See also \cite{Gavruta2001} and \cite{Gavruta1995} for more generalizations. The papers mention above use the direct method (of Hyers), i.e. the exact solution of the functional equation is
explicitly constructed as a limit of a sequence, starting from the given approximate solution.

For other results and  generalizations, see the books \cite{Brzdek2},\cite{Czerwik}, \cite{Hyers1998},\cite{Jung2011} and their references.

Among applications of the functional equations, we mention modeling in science and engineering (see \cite{Castillo}). An interesting application of the stability in the sense of Hyers-Ulam  pointed out by D.H. Hyers, G. Isac and Th. M. Rassias is in the study of complementarity problems (see the book \cite{Hyers1998}). For  other complementarity problems, see also the article of G. Isac \cite{Isac}.\\

In 1991 J.A. Baker \cite{Baker91} used the Banach fixed point theorem to
give Hyers-Ulam stability results for a nonlinear functional equation.

In 2003, V. Radu \cite{RaduCj03}  proposed  a  new  method,  successively  developed  in  \cite{CadRadJipam}, to  obtain the  existence of  the  exact  solutions  and  the  error  estimations,  based  on  the  fixed  point  alternative. For some other applications of the fixed point theorem in the generalized Hyers-Ulam stability see the papers \cite{Cadariu2010-FPTA}, \cite{Cadariu2011Nonlinear}, \cite{CadariuRadu_2012_carpathian}, \cite{Cadariu-Gavruta}, \cite{Cieplinski}, \cite{GavrutaL}, \cite{LGavruta2010}, \cite{Mihet}.

Recently, J. Brzd\c{e}k, J. Chudziak  \& Z. P\'{a}les proved in \cite{Brzdek1} a general fixed point theorem for (not necessarily) linear operators and they used it to obtain  Hyers-Ulam stability results for
a class of functional equations in a single variable. A fixed point result of the same type was proved by J. Brzd\c{e}k $\&$ K. Ciepli\'{n}ski  \cite{BrzdekCieplinski2} in complete non-Archimedean metric spaces as well as in complete metric spaces. Also, they formulated an open problem concerning the uniqueness of the fixed point.

In the paper, \cite{Cadariu&Gavruta2012} we obtained a fixed point theorem for a class of operators with suitable properties, in very general conditions. Also, we showed that some recent results in \cite{Brzdek1} and \cite{BrzdekCieplinski2} can be obtained as particular cases of our theorem. Moreover, by using our outcome, we gave affirmative answer to the open problem of  J. Brzd\c{e}k \& K. Ciepli\'{n}ski, posed in the end of the paper \cite{BrzdekCieplinski2}. We also showed that our main Theorem is an efficient tool for proving generalized Hyers-Ulam stability results of several functional equations in a single variable.\\
\\

In this paper, we prove that an interesting result concerning generalized Hyers-Ulam-Rassias stability of a linear functional equation obtained in 2014 by S.M. Jung, D. Popa and M.T. Rassias in the paper \cite{Jung&Popa&Rassias2014} is a particular case of a fixed point theorem given by us in \cite{Cadariu-Gavruta}. Moreover, we give a characterization of the functions that can be approximated with a given error, by the solution of the previously mention linear equation. \\

We consider a nonempty set $X$, a complete metric space $(Y,d)$  and the mappings $$\Lambda:\mathbb{R}_{+}^{X}\rightarrow \mathbb{R}_{+}^{X}~\textrm{and}~ \mathcal{T}:Y^{X}\rightarrow Y^{X}.$$ We recall that, for two set $M$ and $N$, $N^M$ is the space of all mappings from $M$ to $N$ and if $(\delta_n)_{n \in \mathbb{N}}$ is a sequence of elements of $\mathbb{R}_{+}^{X}$, we write
 $$\lim_{n \rightarrow \infty}\delta_n = 0 \quad \textrm{pointwise if }  \quad \lim_{n \rightarrow \infty} \delta_n(x) = 0 \ \textrm{for every} ~ x \in X.$$

$\mathbb{R}_+$ stands for the set of all nonnegative numbers, i.e. $\mathbb{R}_+=[0,\infty)$ and $\mathbb{R}^*_+=(0,\infty).$

\begin{df} \cite{Cadariu&Gavruta2012}\label{definition_1} We say that  $\mathcal{T}$ is $\Lambda-$ contractive if for all $u,v \in Y^{X}$ and $\delta \in \mathbb{R}_{+}^{X}$ with
$$d(u(x), v(x))\leq \delta(x), \   \forall ~x \in X,$$ it follows
$$d((\mathcal{T}u)(x), (\mathcal{T}v)(x))\leq (\Lambda \delta)(x), \   \forall ~x \in X.$$
 \end{df}

In the paper \cite{Cadariu&Gavruta2012}, we obtained the following fixed point theorem:
\begin{thm}\label{main_theoremCG} We suppose that the operator $\mathcal{T}$ is $\Lambda-$\emph{contractive}, where $\Lambda$ satisfies the condition:\\
$(C_1)$ \quad for every sequence $(\delta_n)_{n \in \mathbb{N}} ~\textrm{in}~ \mathbb{R}_{+}^{X}$ such that
$$\lim_{n \rightarrow \infty} \delta_n = 0 \ \textrm{pointwise, it follows that} \ \lim_{n \rightarrow \infty} \Lambda \delta_n = 0 \ \textrm{pointwise}.$$

 We suppose that $\varepsilon \in \mathbb{R}_{+}^{X}$ is a given function such that $$\varepsilon^{*}(x):=\sum_{k=0}^{\infty}\left(\Lambda^{k} \varepsilon\right)(x) < \infty,  \ \forall x \in X. \leqno(C_2)$$

We consider a mapping $f \in Y^X$ such that
\begin{equation}d((\mathcal{T}f)(x),f(x))\leq \varepsilon(x), \   \forall x \in X.  \label{relatia_1}\end{equation}
Then, for every $x \in X$, the limit
\begin{equation}
g(x):=\lim_{n \rightarrow \infty}(\mathcal{T}^{n} f)(x), \label{limit}
\end{equation} exists and the function $g$ is the unique fixed point of $\mathcal{T}$ with the property
\begin{equation}
d((\mathcal{T}^{m} f)(x),g(x))\leq \sum_{k=m}^{\infty}\left(\Lambda^{k} \varepsilon\right)(t), \ x \in X, \ m\in \mathbb{N}.
\label{relatia_3}
\end{equation}
Moreover, if we have
$$\lim_{n \rightarrow \infty} {\Lambda}^{n} \varepsilon^{*} =0 \ \textrm{pointwise}, \leqno(C_3)\label{C_3} $$
then  $g$ is the unique fixed point of $\mathcal{T}$ with the property
 \begin{equation}
d(f(x), g(x)) \leq \varepsilon^{*}(x), \forall x\in X.\label{relatia_4}
\end{equation}
\end{thm}

 Theorem \ref{main_theoremCG} generalizes a result of J. Brzd\c{e}k and K. Ciepli\'{n}ski \cite{BrzdekCieplinski2} concerning the existence of fixed points. Moreover, our theorem responds to the open question raised by these authors concerning the uniqueness of the fixed point. 

\section{Stability of the functional equation \\ $g(\varphi(x))=a(x)\bullet g(x)$}

We take a nonempty set $X$ and a complete metric group $(G,\bullet, d)$ with the metric $d$ invariant to the left translation, i.e.
$$d(x \bullet y, x\bullet z)=d(y,z), \ \textrm{for all} \  x,y,z \in G.$$

We consider the given functions $\varphi:X \rightarrow X$ and $a:X \rightarrow G$.

We denote $$A_n(x):=  a\left({\varphi}^{n-1}(x)\right) \bullet \ldots \bullet a(\varphi(x)) \bullet a(x), \quad x \in X,\ n \geq 1.$$

We have $$A_n(\varphi(x)) = A_{n+1}(x) \bullet (a(x))^{-1}, \quad x \in X, n \geq 1,$$
and successive $$A_n\left(\varphi^{m}(x)\right) = A_{n+m}(x) \bullet \left(A_{m}(x)\right)^{-1}, \quad x \in X, \ m, n \geq 1.$$

In this section we discuss the generalized Hyers-Ulam-Rassias stability of the functional equation
\begin{equation}g(\varphi(x))=a(x)\bullet g(x), x \in X, \label{ec_generala}\end{equation}
where $g:X \rightarrow G$ is the unknown function.

The equation (\ref{ec_generala}) is equivalent with
\begin{equation} (a(x))^{-1} \bullet g(\varphi(x))= g(x), x \in X. \end{equation}
We remark also that \begin{equation} g\left({\varphi}^{n}(x)\right) = A_{n}(x) \bullet g(x), \quad x \in X, n \geq 1. \label{A_n} \end{equation}

In the following we will show that the main result of the paper \cite{Jung&Popa&Rassias2014} concerning the generalized Hyers-Ulam-Rassias stability of the equation (\ref{ec_generala}) is a simple consequence of our Theorem \ref{main_theoremCG}. To this end, we will start with the presentation of the main result from \cite{Jung&Popa&Rassias2014}:
\begin{thm}\cite{Jung&Popa&Rassias2014}\label{Popa_theorem} Let $\varepsilon: X \rightarrow \mathbb{R}_{+}$ be a given function with the property $$\varepsilon^{*}(x):=\sum_{k=0}^{\infty} \varepsilon\left(\varphi^{k}(x)\right) < \infty,  \ \forall x \in X.$$
Then, for every function $f:X \rightarrow G$ satisfying the inequality
\begin{equation}d(f(\varphi(x)), a(x) \bullet f(x)) \leq \varepsilon(x), \   \forall x \in X, \label{ipoteza_Popa}\end{equation}
there exists a unique solution $g$ of the equation (\ref{ec_generala}) so that
\begin{equation}
d(f(x), g(x)) \leq \varepsilon^{*}(x), \forall x\in X.
\end{equation}
This solution is given by the formula
\begin{equation}
g(x):=\lim_{n \rightarrow \infty}\left(A_{n}(x)\right)^{-1} \bullet f\left(\varphi^{n}(x)\right).
\end{equation}
\end{thm}

We can easily see that the above theorem is a particular case of our fixed point result emphasized in the first section.
\begin{proof} We take in Theorem \ref{main_theoremCG}
$$(\mathcal{T}u)(x)=(a(x))^{-1}\bullet u(\varphi(x)) \quad \textrm{and} \quad (\Lambda \delta)(x)=\delta(\varphi(x)).$$
So, it results
$$d\left((\mathcal{T}u)(x), (\mathcal{T}v)(x)\right)=d(u(\varphi(x)), v(\varphi(x)))\leq (\Lambda \delta)(x)$$
if $$d(u(x), v(x)) \leq  \delta(x),$$ hence the operator $\mathcal{T}$ is $\Lambda-$ contractive in the sense of the Definition \ref{definition_1}.

On the other hand, by using the invariance property to the left translation of the metric $d$ and the assumption (\ref{ipoteza_Popa}), we obtain that (\ref{relatia_1}) holds.

Uniqueness of $g$ results also from Theorem \ref{main_theoremCG}. In fact, we prove that  $\Lambda$  satisfies the hypothesis $(C_3):$
\begin{eqnarray*}\Lambda^n(\varepsilon^{*}(x))&=&\Lambda^n\left(\sum_{k=0}^{\infty}\varepsilon\left(\varphi^{k}(x)\right)\right)= \\ &=&
\sum_{k=0}^{\infty}\varepsilon\left(\varphi^{n+k}(x)\right)=\sum_{m=n}^{\infty}\varepsilon\left(\varphi^{m}(x)\right).\end{eqnarray*}
Thus $$\lim_{n \rightarrow \infty}\Lambda^n(\varepsilon^{*}(x))=0, \ x \in X.$$
\end{proof}

In the second result of this section we will give a characterization of the functions $f:X \rightarrow G$ that can be approximate with a given error,
by a solution of the equation (\ref{ec_generala}).

We denote by $$\mathcal{E}_{\varphi}=\left\{\varepsilon \in \mathbb{R}_{+}^{X},  \ \lim_{n \rightarrow \infty} \varepsilon\left(\varphi^{n}(x)\right)=0, \forall x\in X  \right\}.$$
\begin{thm}\label{second_theorem}
The following statements are equivalent:\\
\\
$(i)$ There exists a unique solution $\ g $ of (\ref{ec_generala}) such that
$$d(f(x), g(x)) \leq \varepsilon (x), \forall x\in X. $$
$(ii) \quad d\left(f\left(\varphi^{n}(x)\right), A_n(x) \bullet f(x) \right) \leq \varepsilon(x)+ \varepsilon\left(\varphi^{n}(x)\right), \ x\in X, n\geq 1. $\\
$(iii) \ \ \textrm{there exists} \quad \delta \in \mathcal{E}_{\varphi} \quad \textrm{so that} $  $$d\left(f\left(\varphi^{n}(x)\right), A_n(x) \bullet f(x) \right) \leq \varepsilon(x)+ \delta\left(\varphi^{n}(x)\right), \  x\in X, n\geq 1.$$
\end{thm}
\begin{proof}$(i) \Rightarrow (ii).$ We have, by using (\ref{A_n})
\begin{eqnarray*}  d\left(f\left(\varphi^{n}(x)\right), A_n(x) \bullet f(x) \right) &\leq & d\left(f\left(\varphi^{n}(x)\right), g\left(\varphi^{n}(x)\right) \right) + d\left(g\left(\varphi^{n}(x)\right), A_n(x) \bullet f(x) \right)   \nonumber \\
&\leq & \varepsilon\left(\varphi^{n}(x)\right) + d\left(A_n(x) \bullet g(x), A_n(x) \bullet f(x) \right)  \\
&=& \varepsilon\left(\varphi^{n}(x)\right) + \varepsilon(x). \nonumber \end{eqnarray*}

$(ii) \Rightarrow (iii).$ We take in $(ii) \ \delta=\varepsilon.$

$(iii) \Rightarrow (i).$ In $(iii)$ with $\varphi^{m}(x)$ instead of $x$, we have
$$ d\left(f\left(\varphi^{n+m}(x)\right), A_n\left(\varphi^{m}(x)\right) \bullet f\left(\varphi^{m}(x)\right)\right) \leq
\varepsilon\left(\varphi^{m}(x)\right)+ \varepsilon\left(\varphi^{n+m}(x)\right),$$
which means
$$ d\left(f\left(\varphi^{n+m}(x)\right), A_{n+m}(x) \bullet (A_{m}(x))^{-1} \bullet f\left(\varphi^{m}(x)\right)\right) \leq
\varepsilon\left(\varphi^{m}(x)\right)+ \varepsilon\left(\varphi^{n+m}(x)\right),$$
hence
$$ d\left( \left(A_{n+m}(x)\right)^{-1}\bullet f\left(\varphi^{m+n}(x)\right) , \left(A_{m}(x)\right)^{-1} \bullet f\left(\varphi^{m}(x)\right) \right) \leq
\varepsilon\left(\varphi^{m}(x)\right)+ \varepsilon\left(\varphi^{n+m}(x)\right).$$
It follows that the sequence $$\left\{ \left(A_{n}(x)\right)^{-1} \bullet f\left(\varphi^{n}(x)\right) \right\}_{n\geq 1}$$
is a Cauchy sequence. Since $(G,\bullet, d)$ is complete, it results that there exists  $$g(x):=\lim_{n \rightarrow \infty} \left(A_{n}(x)\right)^{-1} \bullet f\left(\varphi^{n}(x)\right), x \in X.$$
We have $$g(\varphi(x))=a(x)\bullet \lim_{n \rightarrow \infty} \left(A_{n+1}(x)\right)^{-1} \bullet f\left(\varphi^{n+1}(x)\right)=a(x)\bullet g(x), x \in X,$$
hence $g$ is a solution of (\ref{ec_generala}) and
$$ d\left( g(x), \left(A_{m}(x)\right)^{-1}\bullet f\left(\varphi^{m}(x)\right) \right) \leq
\varepsilon\left(\varphi^{m}(x)\right), x \in X, m\geq 1.$$
By $(iii)$ it follows  that $$d\left(\left(A_{n}(x)\right)^{-1} \bullet f\left(\varphi^{n}(x)\right), f(x)\right) \leq \varepsilon(x)+ \delta\left(\varphi^{n}(x)\right)$$
and by letting $n$ go to infty, we obtain
$$d(f(x), g(x)) \leq \varepsilon (x), \forall x\in X.$$

We prove now the uniqueness of $\ g$. To this end, let us consider $h:X \rightarrow G$ a solution of the equation (\ref{ec_generala}), satisfying the relation
    $$d(h(x), f(x)) \leq \varepsilon (x), \forall x\in X. $$ By replacing $x$ by $\varphi^m(x)$, we have
    $$d( h\left(\varphi^{m}(x)\right), f\left(\varphi^{m}(x)\right)) \leq \varepsilon\left(\varphi^{m}(x)\right), \forall x\in X. $$
Having in mind that $h\left(\varphi^{m}(x)\right)=A_m(x) \bullet h(x)$, it follows
$$ d\left( h(x), \left(A_{m}(x)\right)^{-1}\bullet f\left(\varphi^{m}(x)\right) \right) \leq
\varepsilon\left(\varphi^{m}(x)\right), x \in X.$$ Letting $m$ go to infty , we obtain
$$d(h(x), g(x)) =0, \forall x\in X.$$
\end{proof}

As a direct application of the Theorem \ref{second_theorem} we will obtain the following result concerning the characterization of the functions $f:\mathbb{R}^{*}_{+} \rightarrow \mathbb{R}$ that can be approximate with a given error, by the solutions of Digamma functional equation
\begin{equation}g(x+1)=g(x)+\frac{1}{x}, \ x \in \mathbb{R}^{*}_{+}. \label{digamma}\end{equation}
\begin{cor}
The following statements are equivalent:\\
$(i)$ There exists a unique solution $\ g $ of (\ref{digamma}) such that
$$|f(x) - g(x)| \leq \varepsilon (x), \forall x\in   \mathbb{R}^{*}_{+}. $$
$$ \left|f(x+n)- f(x)-\displaystyle{\sum\limits_{k=0}^{n-1}\frac{1}{x+k}} \right| \leq \varepsilon(x)+ \varepsilon(x+n), \ x\in  \mathbb{R}^{*}_{+}, n\geq 1. \leqno(ii) $$
$(iii)$ There exists $$\delta \in \mathcal{E}_{\varphi}:=\left\{\varepsilon: X \rightarrow \mathbb{R}_{+},  \ \lim_{n \rightarrow \infty} \varepsilon\left(x+n\right)=0, \forall x\in \mathbb{R}^{*}_{+}\right\}$$  so that
$$ \left|f(x+n)- f(x)-\displaystyle{\sum\limits_{k=0}^{n-1}\frac{1}{x+k}}
\right| \leq \varepsilon(x)+ \delta(x+n), \ x\in  \mathbb{R}^{*}_{+}, n\geq 1.$$
\end{cor}
\begin{proof}
The result follows immediately by taking in Theorem \ref{second_theorem}, $X=\mathbb{R}^{*}_{+}$, $(G, \bullet)=(\mathbb{R}, +)$, $d$ the Euclidean metric on $\mathbb{R}$, $\varphi(x)=x+1$, $\displaystyle a(x)=\frac{1}{x}, x\in \mathbb{R}^{*}_{+}$.
\end{proof}

A more general result, obtained below, is in connection with the recent paper of G.H. Kim and Th. M. Rassias \cite{Kim}. 
\begin{cor} Let $p$ be a positive real number.
The following statements are equivalent:\\
$(i)$ There exists a unique solution $\ g $ of the functional equation $$g(x+p)=g(x)+a(x),~x\in\mathbb{R}^*_+$$ such that
$$|f(x) - g(x)| \leq \varepsilon (x), (\forall)~ x\in   \mathbb{R}^{*}_{+}. $$
$$ \left|f(x+np)- f(x)-\displaystyle{\sum\limits_{k=0}^{n-1}a(x+kp)} \right| \leq \varepsilon(x)+ \varepsilon(x+np), \ x\in  \mathbb{R}^{*}_{+}, n\geq 1. \leqno(ii) $$
$(iii)$ There exists $$\delta \in \mathcal{E}_{\varphi}:=\left\{\varepsilon: X \rightarrow \mathbb{R}_{+},  \ \lim_{n \rightarrow \infty} \varepsilon\left(x+np\right)=0, \forall x\in \mathbb{R}^{*}_{+}\right\}$$  so that
$$ \left|f(x+np)- f(x)-\displaystyle{\sum\limits_{k=0}^{n-1}a(x+kp)}
\right| \leq \varepsilon(x)+ \delta(x+np), \ x\in  \mathbb{R}^{*}_{+}, n\geq 1.$$
\end{cor}

\begin{flushright}
L. C\u adariu, L. Manolescu\\
\textit{
{ \normalsize Department of Mathematics, Politehnica University of Timi\c{s}oara, }\\
{ \normalsize Pia\c{t}a Victoriei no.2, 300006 Timi\c{s}oara, Rom\^{a}nia
}}\\

\hspace{5mm}{ \normalsize \textbf{E-mail}: liviu.cadariu-brailoiu@upt.ro\\ laura.manolescu@upt.ro}
\end{flushright}


\begin{thebibliography}{999}
\bibitem{Aoki1950} Aoki T., On the stability of the linear transformation in
Banach spaces,{\em J. M. Soc. Japan} {\bf 1950}, {\em 2}, 64--66. 
\bibitem{Baker91}
Baker J. A.,  The stability of certain functional equations,
{\em Proc. AMS}  {\bf 1991}, {\em 112}(3), 729--732. 
\bibitem{Brzdek1} Brzd\c{e}k J., Chudziak J., P\'{a}les Z.,  A fixed point approach to stability of functional equations, {\em Nonlinear Analysis - TMA} {\bf 2011}, {\em 74 }, 6728--6732.
\bibitem{BrzdekCieplinski2}  Brzd\c{e}k J., Ciepli\'{n}ski K.,  A fixed point approach to the stability of functional equations in non-Archimedean metric spaces, {\em Nonlinear Analysis - TMA} {\bf 2011} {\em 74}, 6861--6867.
\bibitem{Brzdek2} Brzd\c{e}k, J., Popa, D. Raşa, I., Xu, B, \textit{Ulam Stability of Operators}, Academic Press, Cambridge, MA, USA , 2018.
\bibitem{Cadariu&Gavruta2012}  C\u{a}dariu, L., G\u{a}vru\c{t}a, L.,  G\u{a}vru\c{t}a, P., Fixed points and generalized Hyers-Ulam stability, {\em Abstr. Appl. Anal.}  {\bf 2012}, {\em vol. 2012,} Article ID 712743, 10 pages.
\bibitem{CadRadJipam} C\u{a}dariu, L., Radu, V., Fixed points and the stability of Jensen's functional equation, {\em J. Inequal. Pure Appl. Math} {\bf 2003}, {\em 4}(1) ,
Art. 4.
\bibitem{Cadariu2010-FPTA} C\u{a}dariu L., Radu V., Fixed point methods for the generalized stability of functional equations in a single variable, {\em Fixed Point Theory and Applications} {\bf 2008}, {\em vol. 2018} Article ID 749392,  15 pages.

\bibitem{Cadariu2011Nonlinear} C\u{a}dariu, L., Radu, V., In {\em A general fixed point method for the stability
of Cauchy functional equation in Functional Equations in Mathematical Analysis}; Th. M. Rassias, Th. M., Brzdek, J., Eds.,
Series Springer Optimization and Its Applications {\bf 52}, 2011.
\bibitem{CadariuRadu_2012_carpathian} C\u{a}dariu, L., Radu, V.,
A general fixed point method for the stability of the monomial functional equation, {\em Carpathian J. Math.} {\bf 2012} {\em 28}(1), 25--36.
\bibitem{Cadariu-Gavruta} C\u{a}dariu, L., G\u avru\c ta, L., G\u avru\c ta, P., Weighted space method for the stability of some nonlinear equations,{\em Appl. Anal. Discrete Math.} {\bf 2012}, {\em 6}, 126--139.
\bibitem{Cieplinski} Ciepli\'{n}ski, K., Applications of Fixed Point Theorems to the Hyers-Ulam Stability of Functional Equations - A Survey, {\em Ann. Funct. Anal.} {\bf 2012} {\em 3}(1) 151--164.
\bibitem{Castillo} Castillo, E., Ruiz-Cobo, M.R., \textit{Functional Equations and Modelling in Science and Engineering}, Marcel
Dekker, New York, 1992.
\bibitem{Czerwik} Czerwik, S., \textit{Functional Equations and Inequalities in Several Variables}, World Scientific, New Jersey, 2002.
\bibitem{GavrutaL}  G\u{a}vru\c{t}a, L., Matkowski contractions and Hyers-Ulam stability, {\em Bul. \c{S}t. Univ. "Politehnica"
Timi\c{s}oara}, Seria Mat.-Fiz. {\bf 2008} {\em 53}(67)(2), 32--35.
\bibitem{Gavruta1994} G\u{a}vru\c{t}a, P.,  A generalization of the
Hyers-Ulam-Rassias stability of approximately additive mappings,
{\em J. Math. Anal. Appl.} {\bf 1994}, {\em 184},  431--436.
\bibitem{Gavruta2001} G\u{a}vru\c{t}a, P.,  On a problem of G. Isac and Th. M. Rassias concerning the stability of mappings, {\em J. Math. Anal. Appl.} {\bf 2001} {\em 261} 543--553.
\bibitem{LGavruta2010} G\u avru\c ta, P., G\u avru\c ta, L.,  A new method for the generalized Hyers-Ulam-Rassias stability, {\em Int. J. Nonlinear Anal. Appl.} {\bf 2010} {\em 1}(2), 11--18.
\bibitem{Gavruta1995} G\u avru\c ta, P., Hossu, M., Popescu, D., C\u apr\u au, C.,  On the stability of mappings and an answer to a problem of Th.M. Rassias, {\em Ann. Math. Blaise  Pascal} {\bf 1995}, {\em 2}, 55--60. 
\bibitem{Hyers1941}
Hyers, D. H.,  On the stability of the linear functional
equation, {\em Proc. Natl. Acad. Sci. USA} {\bf 1941}, {\em 27}, 222--224.
\bibitem{Hyers1998} Hyers, D. H., Isac, G., Rassias, Th. M.,  \textit{Stability of
Functional Equations in Several Variables}, Birkh\"auser, Boston, 1998.
 \bibitem{Isac} Isac, G., Nonlinear analysis and complementarity theory, {\em Journal of Global Optimization} {\bf 2008} {\em 40}(1), 129–-146.
\bibitem{Jung2011} Jung, S.-M., \textit{Hyers-Ulam-Rassias stability of functional equations in nonlinear analysis}, Series Springer Optimization and Its Applications 48, Springer, 2011.
\bibitem{Jung&Popa&Rassias2014} Jung, S.- M., Popa, D., Rassias, M.T., On the stability of the linear functional equation in a single variable on complete metric groups,  {\em Journal of Global Optimization} {\bf 2014},  {\bf 59} (1), 165--171.
\bibitem{Kim} Kim, G.H., Rassias, Th.M., On the Stability of the Generalized Psi Functional Equation,{\em Axioms} {\bf 2020},{\em 9(2)}58.
\bibitem{Mihet}  Mihe\c{t}, D.,  The Hyers-Ulam stability for two functional equations in a single variable, {\em Banach J. Math. Anal. Appl} {\bf 2008}, {\em 2}(1), 48--52.
\bibitem{RaduCj03} Radu, V., The fixed point alternative and the stability of
functional equations, {\em Fixed Point Theory} {\bf 2003}, {\em 4}(1), 
91--96.
\bibitem{rassiasJohn}   Rassias, J. M.,  Solution of a problem of Ulam,
 {\em J. Approximation Theory} {\bf 1989}, {\em 57}(3), 268--273.
 \bibitem{TRassias} Rassias, Th. M., On the stability of the linear mapping in
Banach spaces, {\em Proc. Amer. Math. Soc.} {\bf 1978}, {\em 72}, 297--300.
\bibitem{TRassias1} Rassias, Th. M.,  On the stability of functional equations and
a problem of Ulam, {\em Acta Appl. Math.} {\bf 2000}, {\em 62}, 23--130.
\bibitem {Ulam} Ulam, S.M., Problems in Modern Mathematics, Chapter VI, Science Editors, Wiley, New York, 1960.



\end{thebibliography}
\end{document}